\documentclass[12pt]{amsart}
\usepackage{mathrsfs}
\usepackage{amssymb}
\usepackage{verbatim}
\usepackage{hyperref}
\usepackage{color}
\usepackage{amsfonts}
\usepackage{mathrsfs}
\usepackage{amsmath}
\usepackage{amssymb}

\begin{document}
\newtheorem{theorem}{Theorem}
\newtheorem{proposition}[theorem]{Proposition}
\newtheorem{conjecture}[theorem]{Conjecture}
\def\theconjecture{\unskip}
\newtheorem{corollary}[theorem]{Corollary}
\newtheorem{lemma}[theorem]{Lemma}
\newtheorem{sublemma}[theorem]{Sublemma}
\newtheorem{observation}[theorem]{Observation}
\theoremstyle{definition}
\newtheorem{definition}{Definition}
\newtheorem{notation}[definition]{Notation}
\newtheorem{remark}[definition]{Remark}
\newtheorem{question}[definition]{Question}
\newtheorem{questions}[definition]{Questions}
\newtheorem{example}[definition]{Example}
\newtheorem{problem}[definition]{Problem}
\newtheorem{exercise}[definition]{Exercise}

\numberwithin{theorem}{section} \numberwithin{definition}{section}
\numberwithin{equation}{section}

\def\earrow{{\mathbf e}}
\def\rarrow{{\mathbf r}}
\def\uarrow{{\mathbf u}}
\def\varrow{{\mathbf V}}
\def\tpar{T_{\rm par}}
\def\apar{A_{\rm par}}

\def\reals{{\mathbb R}}
\def\torus{{\mathbb T}}
\def\heis{{\mathbb H}}
\def\integers{{\mathbb Z}}
\def\naturals{{\mathbb N}}
\def\complex{{\mathbb C}\/}
\def\distance{\operatorname{distance}\,}
\def\support{\operatorname{support}\,}
\def\dist{\operatorname{dist}\,}
\def\Span{\operatorname{span}\,}
\def\degree{\operatorname{degree}\,}
\def\kernel{\operatorname{kernel}\,}
\def\dim{\operatorname{dim}\,}
\def\codim{\operatorname{codim}}
\def\trace{\operatorname{trace\,}}
\def\Span{\operatorname{span}\,}
\def\dimension{\operatorname{dimension}\,}
\def\codimension{\operatorname{codimension}\,}
\def\nullspace{\scriptk}
\def\kernel{\operatorname{Ker}}
\def\ZZ{ {\mathbb Z} }
\def\p{\partial}
\def\rp{{ ^{-1} }}
\def\Re{\operatorname{Re\,} }
\def\Im{\operatorname{Im\,} }
\def\ov{\overline}
\def\eps{\varepsilon}
\def\lt{L^2}
\def\diver{\operatorname{div}}
\def\curl{\operatorname{curl}}
\def\etta{\eta}
\newcommand{\norm}[1]{ \|  #1 \|}
\def\expect{\mathbb E}
\def\bull{$\bullet$\ }

\def\xone{x_1}
\def\xtwo{x_2}
\def\xq{x_2+x_1^2}
\newcommand{\abr}[1]{ \langle  #1 \rangle}

\newcommand{\Norm}[1]{ \left\|  #1 \right\| }
\newcommand{\set}[1]{ \left\{ #1 \right\} }
\def\one{\mathbf 1}
\def\whole{\mathbf V}
\newcommand{\modulo}[2]{[#1]_{#2}}
\def \essinf{\mathop{\rm essinf}}
\def\scriptf{{\mathcal F}}
\def\scriptg{{\mathcal G}}
\def\scriptm{{\mathcal M}}
\def\scriptb{{\mathcal B}}
\def\scriptc{{\mathcal C}}
\def\scriptt{{\mathcal T}}
\def\scripti{{\mathcal I}}
\def\scripte{{\mathcal E}}
\def\scriptv{{\mathcal V}}
\def\scriptw{{\mathcal W}}
\def\scriptu{{\mathcal U}}
\def\scriptS{{\mathcal S}}
\def\scripta{{\mathcal A}}
\def\scriptr{{\mathcal R}}
\def\scripto{{\mathcal O}}
\def\scripth{{\mathcal H}}
\def\scriptd{{\mathcal D}}
\def\scriptl{{\mathcal L}}
\def\scriptn{{\mathcal N}}
\def\scriptp{{\mathcal P}}
\def\scriptk{{\mathcal K}}
\def\frakv{{\mathfrak V}}
\def\C{\mathbb{C}}
\def\R{\mathbb{R}}
\def\Rn{{\mathbb{R}^n}}
\def\Sn{{{S}^{n-1}}}
\def\M{\mathbb{M}}
\def\N{\mathbb{N}}
\def\Q{{\mathbb{Q}}}
\def\Z{\mathbb{Z}}
\def\F{\mathcal{F}}
\def\L{\mathcal{L}}
\def\S{\mathcal{S}}
\def\supp{\operatorname{supp}}
\def\dist{\operatorname{dist}}
\def\essi{\operatornamewithlimits{ess\,inf}}
\def\esss{\operatornamewithlimits{ess\,sup}}
\author{Mingming Cao}
\address{Mingming Cao
\\
School of Mathematical Sciences
\\
Beijing Normal University
\\
Laboratory of Mathematics and Complex Systems
\\
Ministry of Education
\\
Beijing 100875
\\
People's Republic of China
}
\email{m.cao@mail.bnu.edu.cn}

\author{Qingying Xue}
\address{
        Qingying Xue\\
        School of Mathematical Sciences\\
        Beijing Normal University \\
        Laboratory of Mathematics and Complex Systems\\
        Ministry of Education\\
        Beijing 100875\\
        People's Republic of China}
\email{qyxue@bnu.edu.cn}
\author{K\^{o}z\^{o} Yabuta}
\address{K\^{o}z\^{o} Yabuta\\Research center for Mathematical
Science \\Kwansei Gakuin University\\Gakuen 2-1, Sanda 669-1337\\
Japan }
\thanks{The second author was supported partly by NSFC
(No. 11471041), the Fundamental Research Funds for the Central Universities (No. 2012CXQT09) and NCET-13-0065. The third named author was supported partly by
Grant-in-Aid for Scientific Research (C) Nr. 23540228, Japan Society
for the Promotion of Science.\\ \indent Corresponding
author: Qingying Xue\indent Email: qyxue@bnu.edu.cn}

\keywords{Multilinear fractional strong maximal operators; $A_{(\vec{p},q),\mathcal{R}}$ weights; dyadic reverse doubling condition; two-weight inequalities}

\date{August 12, 2014.}
\title[On multilinear fractional strong maximal operator]{\textbf{On multilinear fractional strong maximal operator associated with rectangles and multiple weights}}
\maketitle

\begin{abstract}
In this paper, the multilinear fractional strong maximal operator
$\mathcal{M}_{\mathcal{R},\alpha}$ associated with rectangles and
corresponding multiple weights $A_{(\vec{p},q),\mathcal{R}}$ are introduced.
Under the dyadic reverse doubling condition, a necessary and sufficient
condition for two-weight inequalities is given. As consequences, we first
obtain a necessary and sufficient condition for one-weight inequalities.
Then, we give a new proof for the weighted estimates of multilinear fractional
maximal operator $\mathcal{M}_\alpha$ associated with cubes and multilinear
fractional integral operator $\mathcal{I}_{\alpha}$, which is quite different
and simple from the proof known before.
\end{abstract}


\section{Introduction}
The multilinear Calder\'{o}n-Zygmund theory was originated in the works of
Coifman and Meyer on the Calder\'{o}n-Zygmund commutator \cite{Coifman1},
\cite{Coifman2} in the 70s. Later on, it systematically was studied by
Grafakos and Torres in \cite{GT1}, \cite{GT2}. In recent years, the theory on
multilinear Calder\'{o}n-Zygmund operators and related operators,
such as multilinear singular integral, maximal, strong maximal and fractional
maximal type operators, fractional integrals,
have attracted much attentions as a rapid developing field in harmonic analysis.
\cite{GLPT}, \cite{GT1}, \cite{G}, \cite{Kenig}, \cite{LOPTT} are some important papers on multilinear operators.

In 2009, the following multiple weights class $A_{(\vec{p},q)}$ was first introduced and stydied by Chen and Xue in
\cite{CX}, and also simultaneously defined and studied by Moen in \cite{Moen}.
\begin{definition}(\cite{CX} or \cite{Moen})
Let $1<p_1,\cdots,p_m<\infty$, $\frac{1}{p}=\frac{1}{p_{1}}+\cdots+\frac{1}{p_{m}}$, and $q>0$.
Suppose that $\vec{\omega}=(\omega_1,\cdots,\omega_m)$ and
each $\omega_i$ $(i=1,\cdots,m)$ is a nonnegative function on $\Rn$.
We say that $\vec{\omega} \in A_{(\vec{p},q)}$, if it satisfies
$$\sup_Q \left(\frac{1}{|Q|}\int_Q \nu_{\vec{\omega}}^q dx\right)^{\frac{1}{q}}
\prod_{i=1}^m\left(\frac{1}{|Q|} \int_Q \omega_i^{-p_i'} dx \right)^{\frac{1}{p_i'}}< \infty,$$
where the supremum is taken over all cubes $Q$ with sides parallel to the coordinate axes,
and $\nu_{\vec{w}}=\prod_{i=1}^m \omega_i$. If $p_i=1$, $(\frac1{Q}\int_Q \omega_i^{1-p_i'})^{\frac1{p_i'}}$ is understood as $(\inf_Q \omega_i)^{-1}$.
\end{definition}

Based on a characterization of multiple $A_{(\vec{p},q)}$ weights, \cite{CX} and \cite{Moen} established some weighted estimates for the operators in the following definition.
\begin{definition}(\cite{CX} or \cite{Moen})\label{def 1.2}
Given $\vec{f}=(f_1,\cdots,f_m)$, suppose each $f_i$ $(i=1,\cdots ,m)$ is locally integrable on $\R^n$. Then for any $x\in \R^n$, we define the multilinear fractional type maximal operator $\mathcal{M}_\alpha$ and the multilinear fractional integral operator $\mathcal{I}_{\alpha}$ by
\begin{align}\label{definition of M alpha}
\mathcal{M}_\alpha(\vec{f})(x)=\sup_{Q\ni x} \prod_{i=1}^m \frac{1}{|Q|^{1-\frac{\alpha}{mn}}} \int_Q |f_i(y_i)| \;dy_i,   \quad\quad \text{for} \ \ 0< \alpha<mn\end{align}
and
\begin{equation}
\mathcal{I}_{\alpha}(\vec{f})(x) = \int_{{(\mathbb{R}^n)}^m}\frac{\prod_{i = 1}^{m}f_i(x - y_i)}{{|(y_1, \cdots, y_m)|}^{mn - \alpha}}
\,d\vec{y},\quad\quad \text{for} \ \ 0< \alpha<mn,
\end{equation}
where the supremum in (\ref{definition of M alpha}) is taken over all cubes $Q$ containing $x$ in $\R^n$ with the sides parallel to the axes and $d\vec{y} = dy_1\cdots dy_m$, $|(y_1, \cdots, y_m)| = |y_1| + \cdots + |y_m|$.
\end{definition}

We summarize some results of $\mathcal{M}_\alpha $ and $\mathcal{I}_\alpha $ as follows.

\vspace{0.2cm}
\noindent\textbf{Theorem A} (\cite{Moen})
Let $0<\alpha<mn$, $1< p_1,\cdots,p_m < \infty$, $\frac{1}{p}=\frac{1}{p_{1}}+\cdots+\frac{1}{p_{m}}$ and
$\frac{1}{q}=\frac{1}{p}-\frac{\alpha}{n}$. Then $\vec{\omega}\in A_{(\vec{p},q)} $ if and only if either of the following two inequalities holds,
\begin{equation}
{\big\|\mathcal{M}_\alpha (\vec{f})\big\|}_{L^q({\nu_{\vec{\omega} }}^q)} \leq C
\prod_{i = 1}^m{\big\|f_i\big\|}_{L^{p_i}({\omega_i}^{p_i})};
\end{equation}
\begin{equation}
{\big\|\mathcal{I}_\alpha (\vec{f})\big\|}_{L^q({\nu_{\vec{\omega} }}^q)} \leq C
\prod_{i = 1}^m{\big\|f_i\big\|}_{L^{p_i}({\omega_i}^{p_i})}.
\end{equation}

\vspace{0.2cm}
\noindent\textbf{Theorem B} (\cite{CX} or \cite{Moen})
Let $0<\alpha<mn$, $1\leq p_1,\cdots,p_m < \infty$, $\frac{1}{p}=\frac{1}{p_{1}}+\cdots+\frac{1}{p_{m}}$ and
$\frac{1}{q}=\frac{1}{p}-\frac{\alpha}{n}$. Then for $\vec{\omega}\in A_{(\vec{p},q)}$ there is a constant $C>0$
independent of $\vec{f}$ such that
\begin{equation}
{\big\|\mathcal{M}_\alpha(\vec{f})\big\|}_{L^{q,\infty}({\nu_{\vec{\omega} }}^q)} \leq C
\prod_{i = 1}^m{\big\|f_i\big\|}_{L^{p_i}({\omega_i}^{p_i})};
\end{equation}
\begin{equation}
{\big\|\mathcal{I}_\alpha(\vec{f})\big\|}_{L^{q,\infty}({\nu_{\vec{\omega} }}^q)} \leq C
\prod_{i = 1}^m{\big\|f_i\big\|}_{L^{p_i}({\omega_i}^{p_i})}.
\end{equation}

It is well known that the geometry of rectangles in $\Rn$ is more intricate
than that of cubes in $\Rn$, even when both classes of sets are restricted to
have sides parallel to the axes. This makes the investigation of the strong
maximal function pretty much complex, but of course, quite interesting.
In \rm{1935}, a maximal theorem was given by Jessen, Marcinkiewicz and Zygmund
in \cite{JMZ}. They pointed the strong maximal function is not of weak type
$(1,1)$, which is different from the classical Hardy-Littlewood maximal
operator.
Later, in \rm{1975}, the maximal theorem was again proved by C\'{o}rdoba and
Fefferman applying an alternative geometric method in \cite{CF}. Delicate
properties of rectangles in $\Rn$ were also quantified. The heart of the work
of C\'{o}rdoba and Fefferman is a selection theorem for families of rectangles
in $\Rn$. Their covering lemma is quite useful to study the strong maximal
function, such as \cite{B}, \cite{BK}, \cite{GLPT}, \cite{LS}, \cite{LP}.

Recently, Grafakos, Liu, P\'{e}rez and Torres \cite{GLPT} introduced the multilinear strong maximal function $\mathcal{M}_{\mathcal{R}}$ by setting
$$ \mathcal{M}_{\mathcal{R}}\vec{f}(x)=
\sup_{\substack{R \ni x \\ R\in\mathcal{R}}}\prod_{i=1}^m\frac{1}{|R|} \int_R |f_i(y_i)|dy_i,$$
where $\vec{f}=(f_1,\cdots,f_m)$ is an m-dimensional vector of locally integrable functions and $\mathcal{R}$ denotes the family of all rectangles in $\Rn$ with sides parallel to the axes. Moreover, they gave the definition of the corresponding multiple weights $A_{\vec{p},\mathcal{R}}$ associated with $\mathcal{R}$, where
$\vec{\omega}=(\omega_1,\cdots,\omega_m)\in A_{\vec{p},\mathcal{R}}$ if and only if
$$\sup_{R\in\mathcal{R}} \left(\frac{1}{|R|}\int_R \nu_{\vec{\omega}} dx \right)
\prod_{i=1}^m \left(\frac{1}{|R|} \int_R \omega_i^{1-p_i'} dx \right)^{\frac{p}{p_i'}} < \infty,$$
where $\nu_{\vec{\omega}}=\prod_{i=1}^m \omega_i^{p/{p_i'}}$.
For one-weight case $\vec{\omega}$, the weak and strong type boundedness of the multilinear strong maximal operators
were obtained. For two-weight case $(\vec{\omega},\nu)$, the weak boundedness was established whenever $(\vec{\omega},\nu)$ satisfy a certain power bump variant of the multilinear $A_p$ condition. Moreover, a sharp endpoint distributional unweighted estimate for the multilinear strong maximal operator was given.

\vspace{0.2cm}
\noindent\textbf{Theorem C} (\cite{GLPT})
Let $\vec{p}=(p_1,\cdots,p_m)$ with $1<p_{1},p_2,\cdots,p_m<\infty$, $\frac{1}{p}=\frac{1}{p_{1}}+\cdots+\frac{1}{p_{m}}$
and let $\vec{\omega}$ be an m-tuple of weights. Then
$\vec{\omega}\in A_{\vec{p},\mathcal{R}}$ if and only if one of the following two inequalities holds:
\begin{equation}
{\big\|\mathcal{M}_\mathcal{R}(\vec{f})\big\|}_{L^{p,\infty}(\nu_{\vec{\omega}})} \leq C
\prod_{i = 1}^m{\big\|f_i\big\|}_{L^{p_i}(\omega_i)};
\end{equation}

\begin{equation}
{\big\|\mathcal{M}_\mathcal{R}(\vec{f})\big\|}_{L^p(\nu_{\vec{\omega}})} \leq C
\prod_{i = 1}^m{\big\|f_i\big\|}_{L^{p_i}(\omega_i)}.
\end{equation}

Motivated by the works in \cite{CX}, \cite{GLPT}, \cite{Moen}, \cite{Ko}, we first define the multilinear  fractional strong maximal operators $\mathcal{M}_{\mathcal{R},\alpha}$ and a class of multiple fractional type weights $A_{({\vec{p}},q),\mathcal{R}}$ associated with $\mathcal{R}$, which is the family of all rectangles in $\Rn$ with sides parallel to the coordinate axes. Next we establish some weighted theory for multilinear fractional strong maximal operators. More precisely, not only one-weight inequalities but also two-weight inequalities are obtained. It's worth noting that all arguments for $\mathcal{M}_{\mathcal{R},\alpha}$ are appropriate to  multilinear fractional maximal operator $\mathcal{M}_{\alpha}$ as well. We will employ a different method to show the above strong type inequalities  in Theorem A, compared with the method used by Moen in \cite{Moen}.

The article is organized as follows. Definitions and main results will be listed in Section 2. The proof of the two-weight theorems are given in Section 3. In Section 4, first, a characterization of $A_{({\vec{p}},q),\mathcal{R}}$ can be found. Secondly, the relationship between the weights $A_{p,\mathcal{R}}$, $ A_{p,\mathcal{R}}^d$ and the dyadic reverse doubling condition are given in Proposition \ref{pro 4.2}. Third, the proof of the one-weight theorems are presented. Finally, in Section 5, an alternate simple proof of one-weight estimate of multilinear fractional maximal operator is presented.


\section{Definitions and main results}

\begin{definition}[\textbf{Multilinear fractional strong maximal operator}]
For $ 0< \alpha < mn $, and $\vec{f}=(f_1,\cdots,f_m)\in L_{loc}^1 \times \cdots \times L_{loc}^1 $,
we define the multilinear fractional strong maximal operator $\mathcal{M}_{\mathcal{R},\alpha}$ by
$$ \mathcal{M}_{\mathcal{R},\alpha}\vec{f}(x)=
\sup_{R \ni x}\prod_{i=1}^m\frac{1}{|R|^{1-\frac{\alpha}{mn}}} \int_R |f_i(y_i)|dy_i, \ \ x\in \Rn$$
where the supremum is taken over all rectangles $R$ containing $x$ with sides parallel to the coordinate axes.

Similarly, we can define the dyadic version multilinear fractional strong maximal operator $\mathcal{M}_{\mathcal{R},\alpha}^d$.
\end{definition}

\begin{definition}[\textbf{Class of $A_{(\vec{p},q),\mathcal{R}}$}]
Let$1<p_1,\cdots,p_m<\infty$, $\frac{1}{p}=\frac{1}{p_{1}}+\cdots+\frac{1}{p_{m}}$, and $q>0$.
Suppose that $\vec{\omega}=(\omega_1,\cdots,\omega_m)$ and
each $\omega_i$ $(i=1,\cdots,m)$ is a nonnegative function on $\Rn$.
We say that $\vec{\omega}$ satisfies the $A_{({\vec{p}},q),\mathcal{R}}$ condition or $\vec{\omega} \in A_{(\vec{p},q),\mathcal{R}}$, if it satisfies
$$\sup_R \left(\frac{1}{|R|}\int_R \nu_{\vec{\omega}}^q dx\right)^{\frac{1}{q}}
\prod_{i=1}^m\left(\frac{1}{|R|} \int_R \omega_i^{-p_i'} dx \right)^{\frac{1}{p_i'}}< \infty,$$
where $\nu_{\vec{\omega}}=\prod_{i=1}^m \omega_i$. If $p_i=1$, $(\frac1{R}\int_R \omega_i^{1-p_i'})^{\frac1{p_i'}}$ is understood as $(\inf_R \omega_i)^{-1}$.
We also denote by $A_{(\vec{p},q),\mathcal{R}}^d$ the dyadic analog.
\end{definition}

Throughout the article, we denote by $\mathcal{DR}$ the family of all dyadic rectangles in $\Rn$
with sides parallel to the axes.

\begin{definition}[\textbf{Dyadic reverse doubling condition}]
We say a nonnegative measurable function $\omega$ satisfies the dyadic reverse doubling condition,
or $\omega \in RD^{(d)}$, if $\omega$ is locally integrable on $\Rn$ and there is a constant $d>1$ such that
$$
d \int_I \omega(x)dx \leq \int_J \omega(x)dx
$$
for any $I,J\in \mathcal{DR}$, where $I \subset J$ and $|I|=\frac{1}{2^n}|J|$.
\end{definition}

\noindent{\bf {Remark 2.1.}}
The definition, dyadic reverse doubling condition or reverse doubling condition associated with cubes, can be found in \cite{GGKK}, \cite{KM}, \cite{Tachizawa}. In \cite{GGKK}, it was introduced to investigate the boundedness of the dyadic fractional maximal function. Moreover, in Proposition \ref{pro 4.2} we shall see this condition is very weak.\\

Here we formulate the main results of this paper as follows.

\begin{theorem}[\textbf{Two-weighted  estimates for $\mathcal{M}_{\mathcal{R},\alpha}$}]\label{two-weighted estimates}
Let $0<\alpha<mn$, $\frac{1}{p}=\frac{1}{p_{1}}+\cdots+\frac{1}{p_{m}}$ with $1<p_{1},\cdots,p_m<\infty$,
and $0<p<q<\infty$. Assume that  $\nu$ is an arbitrary weight and for $i=1,\dots,m$,
$\omega_i^{1-p_i'}$ satisfies the dyadic reverse doubling condition. Then $(\vec{\omega},{\nu})$ are weights that satisfy
\begin{equation}\label{weight condition}
\sup_{R\in \mathcal{R}} |R|^{\alpha/n+1/q-1/p}\left(\frac{1}{|R|}\int_R \nu dx \right)^{\frac{1}{q}}
\prod_{i=1}^m\left(\frac{1}{|R|}\int_R \omega_i^{1-p_i'}dx \right)^{\frac{1}{p_i'}}< \infty,
\end{equation}

if and only if either of the following two inequalities holds:
\begin{equation}\label{weak inequality}
{\big\|\mathcal{M}_{\mathcal{R},\alpha}(\vec{f})\big\|}_{L^{q,\infty}({\nu})} \leq C
\prod_{i = 1}^m{\big\|f_i\big\|}_{L^{p_i}(\omega_i)};
\end{equation}

\begin{equation}\label{strong inequality}
{\big\|\mathcal{M}_{\mathcal{R},\alpha}(\vec{f})\big\|}_{L^q({\nu})} \leq C
\prod_{i = 1}^m{\big\|f_i\big\|}_{L^{p_i}(\omega_i)}.
\end{equation}

Similar results hold for $\mathcal{M}_{\mathcal{R},\alpha}^d$ and corresponding dyadic version of two-weight condition
$(\ref{weight condition})$.
\end{theorem}

\begin{theorem}[\textbf{One-weighted  estimates for $\mathcal{M}_{\mathcal{R},\alpha}$}]\label{one-weighted estimates}
Let $0<\alpha<mn$, $\frac{1}{p}=\frac{1}{p_{1}}+\cdots+\frac{1}{p_{m}}$ with $1<p_{1},\cdots,p_m<\infty$,
 and $p<q<\infty$ satisfying $\frac{1}{q}=\frac{1}{p}-\frac{\alpha}{n}$.
Then $\vec{\omega} \in A_{({\vec{p}},q),\mathcal{R}}$ if and only if
either of the following two inequalities holds:
\begin{equation}
{\big\|\mathcal{M}_{\mathcal{R},\alpha}(\vec{f})\big\|}_{L^{q,\infty}({\nu_{\vec{\omega} }}^q)} \leq C
\prod_{i = 1}^m{\big\|f_i\big\|}_{L^{p_i}({\omega_i}^{p_i})};
\end{equation}

\begin{equation}
{\big\|\mathcal{M}_{\mathcal{R},\alpha}(\vec{f})\big\|}_{L^q({\nu_{\vec{\omega} }}^q)} \leq C
\prod_{i = 1}^m{\big\|f_i\big\|}_{L^{p_i}({\omega_i}^{p_i})}.
\end{equation}

Similar results hold for $\mathcal{M}_{\mathcal{R},\alpha}^d$ and $A_{({\vec{p}},q),\mathcal{R}}^d$.
\end{theorem}


\section{Proof of Theorem 2.1}

To prove Theorem \ref{two-weighted estimates}, we need the following lemmas.
\begin{lemma}[\cite{Tachizawa}]\label{lemma elementary inequality}
Let $n$ be a positive integer, $1<p<q<\infty$, and $0<b<2^n$.
Let$$D=\{(F,f,\nu);F\geq 0, \nu > 0, 0\leq f \leq F^{1/p} \nu^{1/{p'}} \}.$$
Then there is a positive constant $ C $ such that
\begin{equation}\label{elementary inequality}
\left(F-\frac{f^p}{2{\nu}^{p/{p^{'}}}}\right)^{q/p}\geq C\frac{f^q}{{\nu}^{q/{p^{'}}}}+
\frac{1}{2^{nq/p}} \sum_{i=1}^{2^n}\left(F_i-\frac{f_i^p}{2{\nu_i}^{p/{p'}}}\right)^{q/p}
\end{equation}
for all $(F,f,\nu),(F_i,f_i,\nu_i)\in D, i=1,\cdots,2^n$,
such that $$  F=\frac{1}{2^n}(F_1+\cdots+F_{2^n}),\ \
f=\frac{1}{2^n}(f_1+\cdots+f_{2^n}),\ \
\nu=\frac{1}{2^n}(\nu_1+\cdots+\nu_{2^n}) $$
 and $\nu_i\leq b\nu$, $i=1,\cdots,2^n$.
\end{lemma}

\noindent{\bf {Remark 3.1.}}
When $1<p,q<\infty$ and $p \geq q$, the above inequality (\ref{elementary inequality}) doesn't hold.
In fact, when $n=1$, we only need to take $F_i=f_i=\nu_i=1$, $i=1,2$.

In order to establish two-weight estimates for multilinear fractional strong maximal operator
we need the next Carleson embedding theorem regarding dyadic rectangles, which is crucial to
the proof of main results.

\begin{lemma}[Carleson Embedding Theorem Regarding Dyadic Rectangles]\label{Carleson embedding theorem}
Let $1<p<q<\infty$, $\omega$ be a nonnegative locally integrable function on $\mathbb{R}^n$.
Assume that $\omega^{1-p'}$ satisfies the dyadic reverse doubling condition.
Then the inequality
$$\sum_{I\in \mathcal{DR}} \left(\int_I {\omega}^{1-p'}dx \right)^{-q/{p'}} \left(\int_I f(x)dx \right)^q
 \leq C \left(\int_{\mathbb{R}^n}f(x)^p \omega dx \right)^{q/p}$$
holds for all nonnegative $f\in L^p(\omega)$.
\end{lemma}

\noindent\textbf{Proof of Lemma 3.2}.
\begin{proof}
The proof of this conclusion is a routine application of the method of Theorem
$1.1$ \cite{Tachizawa}.
However, we here present the proof for the sake of completeness.

For a given $I \in \mathcal{DR}$, it suffices to show
\begin{equation}\label{e 3.2}
\sum_{\substack{J\subset I\\ J\in \mathcal{DR}}} \left(\int_J \omega^{1-p'}dx
\right)^{-q/{p'}} \left( \frac{1}{|J|} \int_J f(x)dx \right)^q
\leq C \left(\int_I f(x)^p \omega dx \right)^{q/p}
\end{equation}
for all nonnegative locally integrable functions $f$, where $C$ is a constant
which does not depend on $I$.

Let $D$ be the domain in Lemma \ref{lemma elementary inequality}. For
$(F,f,\nu)\in D$, we set
$$
B(F,f,\nu)=\frac{1}{c}\left(F-\frac{f^p}{2\nu^{p/{p'}}} \right)^{q/p}
$$
where $c$ is the constant in Lemma \ref{lemma elementary inequality}.
Let $f$ be a nonnegative measurable function such that
$$ \int_I f(x)^p \omega dx < \infty.$$
Denote
$$
F_A=\frac{1}{|A|}\int_A f(x)^p \omega dx,\ \ f_A=\frac{1}{|A|}\int_A f(x)dx
$$
and
$$ \nu_A=\frac{1}{|A|}\int_A \omega(x)^{1-p'}dx,
$$
for a measurable set $A$ in $I$ such that $|A|\neq 0$.
Then, by H\"{o}lder's inequality, we have
$$
\frac{1}{|I|}\int_I f(x)dx
\leq \left(\frac{1}{|I|}\int_I f(x)^p \omega dx \right)^{1/p}
 \left(\frac{1}{|I|}\int_I \omega^{-p'/p}dx \right)^{1/{p'}}  .
$$
That is to say $$ 0 \leq f_I \leq F_I^{1/p}\nu^{1/{p'}}.$$
Then we get $(F_I,f_I,\nu_I)\in D $.
Let $I_1,\cdots, I_{2^n} $ be dyadic rectangles which are obtained by dividing
$I$ into $2^n$ equal parts. Then we have
$$
(F_{I_i},f_{I_i},\nu_{I_i})\in D,\ \ i=1,\cdots,n
$$
$$
F_I=\frac{1}{2^n}(F_{I_1}+\cdots+F_{I_{2^n}}),\ \ f_I
=\frac{1}{2^n}(f_{I_1}+\cdots+f_{I_{2^n}}),
$$
$$
\nu_I=\frac{1}{2^n}(\nu_{I_1}+\cdots+\nu_{I_{2^n}}).
$$
Moreover, by the dyadic reverse doubling condition for $\omega^{1-p'}$, we can
get $\nu_{I_i} \leq b \nu_I$,
$i=1,\cdots,n$, where $b=2^n / d < 2^n $.
Hence, according to Lemma \ref{lemma elementary inequality}, we have
$$
B(F_I,f_I,\nu_I) \geq \frac{f_I^q}{\nu_I^{q/{p'}}}
+\frac{1}{2^{nq/p}} \sum_{i=1}^{2^n}B(F_{I_i},f_{I_i},\nu_{I_i}),
$$
and hence, we obtain the inequality
$$
|I|^{q/p} B(F_I,f_I,\nu_I) \geq |I|^{q/p}\nu_I^{-q/{p'}}f_I^q
+ \sum_{1}^{2^n} |I_i|^{q/p} B(F_{I_i},f_{I_i},\nu_{I_i}).
$$
Using the same technique for $I_i$, $i=1,\cdots,n$, we can get
\begin{align*}
&|I|^{q/p}B(F_I,f_I,\nu_I)
\geq \sum_{\substack{J\subset I,J\in \mathcal{DR}\\ |J|\geq 2^{-nk}|I|}} |J|^{q/p}\nu_J^{-q/{p'}}f_J^q +
\sum_{\substack{J\subset I,J\in \mathcal{DR}\\ |J|=2^{-n(k+1)}|I|}} |J|^{q/p} B(F_J,f_J,\nu_J)\\
&\geq \sum_{\substack{J\subset I,J\in \mathcal{DR}\\ |J|\geq 2^{-nk}|I|}} |J|^{q/p}\nu_J^{-q/{p'}}f_J^q
= \sum_{\substack{J\subset I, J\in \mathcal{DR}\\ |J|\geq 2^{-nk}|I|}} \left(\int_J \omega^{1-p'}dx \right)^{-q/{p'}} \left(\int_J f(x)dx \right)^q .
\end{align*}
Letting $k$ approach $\infty$, we have
\begin{align*}
&\sum_{\substack{J\subset I\\ J\in \mathcal{DR}}} \left(\int_J \omega^{1-p'}dx \right)^{-q/{p'}} \left(\int_J f(x)dx \right)^q \leq |I|^{q/p}B(F_I,f_I,\nu_I)\\
&=|I|^{q/p} \frac{1}{c}\left(F_I-\frac{f_I^p}{2\nu_I^{p/{p'}}} \right)^{q/p} \leq C'|I|^{q/p}F_I^{q/p}=C'\left(\int_I f(x)^p \omega dx \right)^{q/p}.
\end{align*}
\end{proof}

Before proving Theorem \ref{two-weighted estimates}, we also need the following key lemma, which connects
the maximal operator with the corresponding dyadic version maximal operator.
\begin{lemma}\label{lemma truncated dyadic strong maximal operator}
Let $x,t \in \mathbb{R}^n$, $0<\alpha<mn$.
For any $\vec{f}\geq 0$, $ k\geq 0$, we define the following truncated dyadic version strong maximal operator,
\begin{align*}
\mathcal{M}_{\mathcal{R},\alpha}^{(k)}(\vec{f})(x):=
\sup_{\substack{R \ni x, R\in\mathcal R, \\ \text{every side length of }R\le 2^k}}
\prod_{i=1}^m \frac{1}{|R|^{1-\frac{\alpha}{mn}}} \int_R f_i(y_i)dy_i.
\end{align*}
Then we have,
$$ \mathcal{M}_{\mathcal{R},\alpha}^{(k)}(\vec{f})(x)\leq \frac{C_{n,\alpha}}{|B_k|}\int_{B_k}
\tau_{-t} \circ \mathcal{M}_{\mathcal{R},\alpha}^d \circ \vec{\tau}_t (\vec{f})(x) dt,
\ for \ any \ k\geq 0 \ ,$$
where $B_k=[-2^{k+2},2^{k+2}]^n$, $\tau_t g(x) =g(x-t)$, $\vec{\tau}_t \vec{f}=(\tau_t f_1,\cdots,\tau_t f_m)$.
\end{lemma}

We notice that this inequality, for Hardy-Littlewood maximal operators, was
established by Fefferman and Stein (see \cite[p.~431]{GR}). For fractional
maximal operators, it is due to Saywer (see \cite{Sawyer1}, \cite{Sawyer2}).
\begin{proof}
Our method is similar to \cite{KM}.

First we shall need the following observation (see \cite[p.~431]{GR}).
Let $j$ be an integer and $I$ be an interval satisfying $2^{j-1}|I| \leq 2^j$.
Suppose that $k \in \mathbb{Z}$, $j \leq k$.
Denote the set
$$
E:=\{t\in (-2^{k-2},2^{k-2});\exists J \text{ such that } |J|=2^{j+1},
I\subset J, J+t \text{ dyadic }\}.
$$
Then, $|E|\geq2^{k+2}$.\\

For given $\vec{f}\geq 0$ and $x=(x_1,\cdots,x_n)\in \mathbb{R}^n$,
we take interval $I_j\ni x_j$, $|I_j|\leq 2^k$, for all $j=1,\cdots,n$, such that
$$ \mathcal{M}_{\mathcal{R},\alpha}^{(k)}(\vec{f})(x)<
2\prod_{i=1}^m \frac{1}{|R|^{1-\frac{\alpha}{mn}}} \int_R f_i(y_i)dy_i $$
where the rectangle $R=I_1 \times \cdots \times I_n$.
Now, for each $j=1,\cdots,n$, we take integers $n_j$ such that $2^{n_j-1}<|I_j| \leq 2^{n_j}$.
It is obvious that $n_j \leq k$.
In addition, let us denote the sets $E_j$ by the following way:
$$E_j:=\{t\in (-2^{k-2},2^{k-2});there \ exists \ I_j' \  such \ that \ |I_j'|=2^{j+1},
I_j\subset I_j', I_j'+t \ dyadic  \}$$
for each $j$.
Denote $E=E_1 \times \cdots \times E_n$.
Then for each $t=(t_1,\cdots,t_n) \in E$,
there exist intervals $\{I_j'\}_{j=1}^n$ such that $I_j \subset I_j'$ and $I_j'+t_j$ dyadic, $i=1,\cdots,n$.
Let $R^{''}=(I_1'+t_1)\times \cdots \times(I_n'+t_n)$. Therefore $R\subset {R^{''}-t}$ and $R^{''}$ is a dyadic rectangle.
It follows that
\begin{align*}
\mathcal{M}_{\mathcal{R},\alpha}^{(k)}(\vec{f})(x)
&<2\prod_{i=1}^m \frac{1}{|R|^{1-\frac{\alpha}{mn}}} \int_R f_i(y_i)dy_i
\\
&\leq 2\prod_{i=1}^m \left( \frac{|R''|}{|R|}\right)^{1-\frac{\alpha}{mn}}
\frac{1}{|R''|^{1-\frac{\alpha}{mn}}} \int_{R''-t}f_i(y_i)dy_i
\\
&\leq 2\times 4^{mn-\alpha}\sup_{\substack{R-t\ni x\\ R\in \mathcal{DR}}}
\prod_{i=1}^m \frac{1}{|R|^{1-\frac{\alpha}{mn}}} \int_{R-t}f_i(y_i)dy_i
\end{align*}
where for any rectangle $R=I_1 \times \cdots \times I_n$,
we denote $R-t=(I_1-t_1) \times \cdots \times (I_n-t_n)$.
Notice that $E \subset B_k$ and $|E| \geq \frac{1}{2^n}|B_k|$.
Hence,
\begin{align*}
\mathcal{M}_{\mathcal{R},\alpha}^{(k)}(\vec{f})(x)
&\leq \frac{C_{n,\alpha}}{|E|}
\int_E \tau_{-t} \circ \mathcal{M}_{\mathcal{R},\alpha}^d \circ \vec{\tau}_t (\vec{f})(x) dt \\
&\leq 2^n \times \frac{C_{n,\alpha}}{|B_k|}
\int_{B_k} \tau_{-t} \circ \mathcal{M}_{\mathcal{R},\alpha}^d \circ \vec{\tau}_t (\vec{f})(x) dt .
\end{align*}
The lemma has been proved.
\end{proof}

\noindent\textbf{Proof of Theorem 2.1}.
\begin{proof}
In fact, we only need to prove (\ref{strong inequality}) $\Rightarrow$ (\ref{weak inequality})
$\Rightarrow$ (\ref{weight condition}) $\Rightarrow$ (\ref{strong inequality}).

It is obvious that (\ref{strong inequality}) $\Rightarrow$ (\ref{weak inequality}).\\
(\rm{I}). First of all, we will prove (\ref{weak inequality}) implies (\ref{weight condition}).

We can assume that $\vec{f}\geq 0$, for fixed rectangle $R$,
$$\prod_{i=1}^m \frac{1}{|R|^{1-\frac{\alpha}{mn}}} \int_R f_i(y_i)dy_i \geq 0 .$$
For $x\in R$, we have
$$\prod_{i=1}^m \frac{1}{|R|^{1-\frac{\alpha}{mn}}} \int_R f_i(y_i)dy_i
\leq \mathcal{M}_{\mathcal{R},\alpha}(f_1\chi_R,\cdots,f_m\chi_R)(x).$$
Therefore, for any $$ 0 < \lambda < \prod_{i=1}^m \frac{1}{|R|^{1-\frac{\alpha}{mn}}} \int_R f_i(y_i)dy_i,$$
we have $R\subset \{x\in\Rn;\mathcal{M}_{\mathcal{R},\alpha}(f_1\chi_R,\cdots,f_m\chi_R)(x) >\lambda \}$.\\
Hence, by (\ref{weak inequality}) we get
\begin{align*}
\nu(R)
\leq \nu(\{x\in\Rn;\mathcal{M}_{\mathcal{R},\alpha}^d(f_1\chi_R,\cdots,f_m\chi_R)(x) >\lambda \})
\leq \left(\frac{C}{\lambda}\prod_{i = 1}^m{\big\|f_i\big\|}_{L^{p_i}(\omega_i)} \right)^q.
\end{align*}
Letting $$\lambda\rightarrow \prod_{i=1}^m \frac{1}{|R|^{1-\frac{\alpha}{mn}}} \int_R f_i(y_i)dy_i,$$
we obtain
$$\sup_{R\in \mathcal{R}} |R|^{\alpha/n-m}\nu(R)^{\frac{1}{q}}\prod_{i=1}^m \int_R f_i(y_i) dy_i
\leq C \prod_{i = 1}^m{\big\|f_i\big\|}_{L^{p_i}(\omega_i)} .$$
Taking $f_i=\omega_i^{1-p_i'}$, we get
$$\sup_{R\in \mathcal{R}} |R|^{\alpha/n-m}\nu(R)^{\frac{1}{q}} \prod_{i=1}^m \int_R \omega_i^{1-p_i'}dx
\leq C \prod_{i = 1}^m \left(\int_R \omega_i^{1-p_i'} dx \right)^{\frac{1}{p_i}}.$$
Therefore, $(\vec{\omega},\nu)$ satisfies condition (\ref{weight condition}).\\
(\rm{II}). Next, let us prove (\ref{weight condition}) implies (\ref{strong inequality}).

\noindent\textbf{$\bullet$ Estimate for $\mathcal{M}_{\mathcal{R},\alpha}^d$}.
We first prove the boundedness for the dyadic version,
$$
\mathcal{M}_{\mathcal{R},\alpha}^d(\vec{f})(x)=
\sup_{\substack{R \ni x \\ R \in \mathcal{DR}}}\prod_{i=1}^m
\frac{1}{|R|^{1-\frac{\alpha}{mn}}} \int_R |f_i(y_i)|dy_i, \ \ x\in \Rn.
$$
Without loss of generality, we can assume that $\vec{f} \geq 0$ , bounded and
has a compact support.
Therefore $\mathcal{M}_{\mathcal{R},\alpha}^d(\vec{f})(x)< \infty$ for all
$x\in \mathbb{R}^n$. According to the definition of
$\mathcal{M}_{\mathcal{R},\alpha}^d(\vec{f})(x)$, for any $x\in \Rn$,
there exists a dyadic rectangle $R$ such that $x\in R$ and
\begin{equation}\label{equation 3.3}
\mathcal{M}_{\mathcal{R},\alpha}^d(\vec{f})(x) \leq 2 \prod_{i=1}^m \frac{1}{|R|^{1-\frac{\alpha}{mn}}} \int_R f_i(y_i)dy_i.
\end{equation}
For any dyadic rectangle $R$, define the set $E(R)$ by
$$
E(R):=\{x\in \mathbb{R}^n; x\in R \ \text{and} \ R \
\text{is minimal for which} \ (\ref{equation 3.3}) \ \text{holds} \} .
$$
From the definition of maximal operator and the latter inequality it is obvious that
$$
\mathbb{R}^n=\bigcup_{R\in \mathcal{DR}} E(R).
$$
Since $ (\vec{\omega},\nu)$ satisfies the condition (\ref{weight condition}), it follows that
\begin{align*}
&\int_\Rn \left(\mathcal{M}_{\mathcal{R},\alpha}^d(\vec{f})(x)\right)^q \nu dx\\
&\leq \sum_{R\in \mathcal{DR}} \int_{E(R)}\left(\mathcal{M}_{\mathcal{R},\alpha}^d(\vec{f})(x) \right)^q \nu dx\\
&\lesssim \sum_{R\in \mathcal{DR}} \int_R \left(\prod_{i=1}^m \frac{1}{|R|^{1-\frac{\alpha}{mn}}} \int_R f_i(y_i)dy_i \right)^q \nu dx\\
&=\sum_{R\in \mathcal{DR}} \left( \prod_{i=1}^m\left(\int_R \omega_i^{1-p_i'}dx \right)^{-q/{p_i'}} \left(\int_R f_i(y_i)dy_i \right)^q\right)\\
&\quad\quad\quad \times \left( |R|^{\alpha/n+1/q-1/p} \left(\frac{1}{|R|} \int_R \nu dx \right)^{1/q} \prod_{i=1}^m \left(\frac{1}{|R|} \int_R \omega_i^{1-p_i'}dx \right)^{1/{p_i'}}  \right)^q\\
&\lesssim \sum_{R\in \mathcal{DR}}  \prod_{i=1}^m \left(\int_R \omega_i^{1-p_i'}dx \right)^{-q/{p_i'}} \left(\int_R f_i(y_i)dy_i \right)^q .
\end{align*}
By H\"{o}lder's inequality $\sum_{j=1}^\infty \prod_{i=1}^m |a_{ij}|\leq \prod_{i=1}^m (\sum_{j=1}^\infty |a_{ij}|^{p_i/p})^{p/p_i}$ and Lemma \ref{Carleson embedding theorem}, we further deduce that
\begin{align*}
&\int_\Rn \left(\mathcal{M}_{\mathcal{R},\alpha}^d(\vec{f})(x)\right)^q \nu dx\\
&\leq  \prod_{i=1}^m  \bigg[\sum_{R\in \mathcal{DR}}  \left(\int_R \omega_i^{1-p_i'}dx \right)^{(-qp_i)/{(p_i'p)}} \left(\int_R f_i(y_i)dy_i \right)^{qp_i/p}\bigg]^{p/p_i}\\
&\lesssim \prod_{i = 1}^m{\big\|f_i\big\|}_{L^{p_i}({\omega_i})}^q.
\end{align*}

\noindent\textbf{$\bullet$ Estimate for $\mathcal{M}_{\mathcal{R},\alpha}$}.
From Lemma \ref{lemma truncated dyadic strong maximal operator} and generalized Minkowski's inequality, it follows that
\begin{align*}
\big\|\mathcal{M}_{\mathcal{R},\alpha}^{(k)} (\vec{f})\big\|_{L^q(\nu)}
&\lesssim \frac{1}{|B_k|} \big\|\int_{B_k} \tau_{-t} \circ \mathcal{M}_{\mathcal{R},\alpha}^d \circ \vec{\tau}_t (\vec{f}) dt\big\|_{L^q(\nu)}  \\
&\leq \frac{1}{|B_k|} \int_{B_k} \big\| \tau_{-t} \circ \mathcal{M}_{\mathcal{R},\alpha}^d \circ \vec{\tau}_t (\vec{f}) \big\|_{L^q(\nu)} dt \\
&\leq \frac{1}{|B_k|} \int_{B_k} \big\| \mathcal{M}_{\mathcal{R},\alpha}^d \vec{\tau}_t (\vec{f}) \big\|_{L^q(\tau_t \nu)} dt.
\end{align*}
Since $(\vec{\omega},{\nu})$ satisfies the condition (\ref{weight condition}),
we can further verify $(\vec{\tau}_t \vec{\omega},{\tau_t \nu})$ also satisfies
  the condition (\ref{weight condition}) independently of $t$. Therefore, from (\ref{strong inequality}) we deduce that
\begin{align*}
&\big\|\mathcal{M}_{\mathcal{R},\alpha}^{(k)} (\vec{f})\big\|_{L^q(\nu)}
\lesssim \frac{1}{|B_k|} \int_{B_k} \prod_{i=1}^m \big\|  \tau_t f_i \big\|_{L^{p_i}(\tau_t \omega_i)} dt \\
&=\frac{1}{|B_k|} \int_{B_k} \prod_{i=1}^m \big\| f_i \big\|_{L^{p_i}(\omega_i)} dt
=\prod_{i=1}^m \big\| f_i \big\|_{L^{p_i}(\omega_i)}.
\end{align*}
Finally, letting $k$ tend to infinity, we finish the proof.
\end{proof}


\section{Proof of Theorem 2.2}\label{sec2a}

To complete the proof of Theorem \ref{one-weighted estimates}, we need the following characterizations of
$A_{({\vec{p}},q),\mathcal{R}}$ class and the the connection between the weights $ A_{p,\mathcal{R}}^d$ and the dyadic reverse doubling condition.

\begin{proposition}[\textbf{Characterization of $A_{({\vec{p}},q),\mathcal{R}}$ class} ]\label{characterization of weight class}

Let $0<\alpha<mn$, $1<p_1,\cdots,p_m<\infty$,
$\frac{1}{p}=\frac{1}{p_{1}}+\cdots+\frac{1}{p_{m}}$ and
$\frac{1}{q}=\frac{1}{p}-\frac{\alpha}{n}$.
Suppose $\vec{\omega}\in A_{({\vec{p}},q),\mathcal{R}}$, then\\
\begin{enumerate}
\item [(a)]${\nu_{\vec{\omega} }}^q  \in A_{r,\mathcal{R}} \subset A_{mq,\mathcal{R}}$;\\
\item [(b)]$\omega_i^{-p_i'} \in A_{mp_i',\mathcal{R}}$;\\
\item [(c)]$\omega_i^{-p_i'} \in A_{r_i,\mathcal{R}}$,
if $\frac{\alpha}{n}<(m-2)+\frac{1}{p_i}+\frac{1}{p_j}$,
for any $1\leq i,j \leq m,$\\
where $r=1+q(m-\frac{1}{p})$ and $r_i=1+\frac{p_i'}{q}[1+(m-1)q-\frac{q}{p}+\frac{q}{p_i}]$.
\end{enumerate}
\end{proposition}

\begin{proof}
By examining the proof of Theorem 2.2 in \cite{CX}, we shall find the arguments used in \cite{CX} rely only on the use of H\"{o}lder's inequality, and it doesn't involve any geometric property of cubes or rectangles.
Hence we can also adopt the method in \cite{CX} to complete our proof.
Since the main ideas are almost the same, we omit the proof here.
\end{proof}

\begin{proposition}\label{pro 4.2}
Let $1<p<\infty$, $\omega$ is a nonnegative weight. Then we have\\
\begin{enumerate}
\item [(i)]$A_{\infty,\mathcal{R}} \subset A_{\infty,\mathcal{R}}^d \subset RD^{(d)} $;\\
\item [(ii)] if $\omega \in A_{p,\mathcal{R}}^d$, then $\omega^{1-p'} \in RD^{(d)}$.
\end{enumerate}
\end{proposition}

\noindent\textbf{Proof of Theorem 2.2}
\begin{proof}
From Proposition \ref{characterization of weight class} \rm{(b)} and Proposition \ref{pro 4.2} \rm{(i)}, we can get the fact,
if $\vec{\omega}\in A_{({\vec{p}},q),\mathcal{R}}$, then $\omega_i^{-p_i'} \ (i=1,\cdots,m)$
satisfies the dyadic reverse doubling condition.

Now, by means of the above arguments and using $\omega_i^{p_i},\nu_{\vec{\omega}}$ substitute $\omega_i,\nu$ respectively in Theorem \ref{two-weighted estimates}, we can conclude the results of Theorem \ref{one-weighted estimates}.
\end{proof}

We next turn to the proof of Proposition \ref{pro 4.2}.

\vspace{0.2cm}
\noindent\textbf{Proof of Proposition 4.2}
\begin{proof}
According to the definition of these weights, it is easy to see the inclusion
$A_{\infty,\mathcal{R}} \subset A_{\infty,\mathcal{R}}^d $. Now let us prove
{\rm(ii)}.

Let $I$ be any dyadic rectangle.
By dividing $I$ into $2^n$ equal parts, we can get dyadic sub-rectangles of $I$, $I_1,\cdots, I_{2^n} $.
We denote $$ u_A=\frac{1}{|A|} \int_A \omega(x) dx,\ \ \nu_A=\frac{1}{|A|} \int_A \omega(x)^{1-p'} dx,$$
for a measurable set $A \subset I$ and $|A|\neq 0$.
Then we have $$ u_A=\frac{1}{2^n}(u_{I_1}+\cdots+u_{I_{2^n}}),\ \ \nu_A=\frac{1}{2^n}(\nu_{I_1}+\cdots+\nu_{I_{2^n}}).$$
Notice that $u_{I_i}\leq 2^n u_I$, for each $i=1,\cdots,n$.
Since $\omega \in A_{p,\mathcal{R}}^d $, for each $i=1,\cdots,n$, we have
$$1 \leq u_I^{\frac1p} \nu_I^{\frac1{p'}} \leq K,\ \ 1 \leq u_{I_i}^{\frac1p} \nu_{I_i}^{\frac{1}{p'}} \leq K,$$
where $K$ is a positive constant which does not depend on $I$.
Hence we obtain the following inequalities, for all $i=1,\cdots,n$,
$$
\nu_{I_i} \geq \frac{1}{u_{I_i}^{p'/p}} \geq \frac{1}{(2^n u_I)^{p'/p}}
\geq \frac{\nu_I}{2^{np'/p}\cdot K^{p'}} ,
$$
$$
\nu_{I_i}=2^n \nu_I-\sum_{j\neq i}\nu_{I_j}
\leq \left(2^n-\frac{2^n-1}{2^{np'/p}\cdot K^{p'}} \right)\nu_I
=2^n\left(1-\frac{1-2^{-n}}{2^{np'/p}\cdot K^{p'}} \right)\nu_I.
$$
We take $d$ such that $\frac{1}{d}=1-\frac{1-2^{-n}}{2^{np'/p}\cdot K^{p'}}$.
Because $K\geq1$, $p>1$, we have $d>1$.
Finally, we conclude
$$
d \int_{I_i}\omega(x)^{1-p'}dx=d |I_i|\nu_{I_i}
\leq 2^n |I_i| \nu_I=|I|\nu_I=\int_I\omega(x)^{1-p'}dx.
$$
Hence $\omega^{1-p'} \in RD^{(d)}$. This proves (ii).

Next, let $\omega\in A_{p,\mathcal R}^d$. By the definition of $A_{p,\mathcal R}^d$, it can be easily seen that
$\omega^{-\frac{1}{p-1}}\in A_{p',\mathcal R}^d$. Hence we have by \rm{(ii)}
\begin{equation*}
\omega=\bigl(\omega^{-\frac{1}{p-1}}\bigr)^{-\frac{1}{p'-1}}\in RD^{(d)}.
\end{equation*}
This completes the proof of Proposition \ref{pro 4.2}.
\end{proof}

\begin{corollary}\label{cor 4.3}
Let $1<p<q<\infty$, $\omega \in A_{p,\mathcal{R}}^d $.
Let $\{\mu_I\}_{I\in \mathcal{DR}} $ be nonnegative numbers.
Then the following two statements are equivalent:
\begin{enumerate}
\item [(i)] There is a positive constant $C_1$ such that
$$
\sum_{I\in \mathcal{DR}} \mu_I \left(\frac{1}{|I|}\int_I f(x)dx \right)^q
\leq C_1 \left(\int_\Rn f(x)^p \omega(x)dx \right)^{q/p}
$$
for all nonnegative locally integrable function $f$.
\item [(ii)]  There is a positive constant $C_2$ such that
$$
\mu_I \leq C_2 \left(\int_I \omega(x)dx \right)^{q/{p}}
$$
for all $I\in \mathcal{DR}$.
\end{enumerate}
\end{corollary}
\begin{proof}
Taking $f=\chi_I$, we can check that \rm{(i)} implies \rm{(ii)}. Conversely,
from $\omega \in A_{p,\mathcal{R}}^d $ and H\"{o}lder's inequality, we get
\begin{align*}
1=\frac{1}{|I|}\int_I \omega(x)^{\frac1p}\omega(x)^{-\frac1p} dx
\leq \left(\frac{1}{|I|}\int_I \omega(x) dx\right)^{\frac1p}
\left(\frac{1}{|I|}\int_I \omega(x)^{-\frac{1}{p-1}}dx\right)^{\frac{1}{p'}}
\leq C
\end{align*}
for all $I\in \mathcal{DR}$,
where the constant $C$ independent of $I$.
By combining Lemma \ref{Carleson embedding theorem} and Proposition
\ref{pro 4.2}
with the above inequality, we can easily see that \rm{(ii)} implies \rm{(i)}.
\end{proof}


\section{A new proof of multilinear fractional integral operators and maximal operators}

Our goal here is to present a new proof for Theorem A.

In fact, by verifying the proof of Theorem \ref{one-weighted estimates}, we will find that the method  we use is
likewise appropriate for the multilinear fractional maximal operator $\mathcal{M}_\alpha$ in Definition \ref{def 1.2}.
That is to say, the conclusions of Theorem \ref{one-weighted estimates} also hold for $\mathcal{M}_\alpha$ and $ A_{(\vec{p},q)}$,
$\mathcal{M}_\alpha^d$ and $ A_{(\vec{p},q)}^d$.
To complete the weighted estimate for $\mathcal{I}_\alpha$, we only need the following proposition.
\begin{proposition}\label{pro 5.1}
Let $0<q<\infty$, and $0<\alpha<mn$. If $\omega \in A_\infty$,
then there exists a positive constant $C$ independent of $f$ such that
\begin{equation}\label{strong estimate}
\int_{\mathbb{R}^n}|\mathcal{I}_\alpha(\vec{f})(x)|^q\omega(x)dx \leq C
\int_\Rn [\mathcal{M}_\alpha(\vec{f})(x)]^q\omega(x)dx
\end{equation}
and
\begin{equation}\label{weak estimate}
\sup_{\lambda>0} \lambda^q\omega(x\in \mathbb{R}^n;|\mathcal{I}_\alpha(\vec{f})(x)|>\lambda) \leq C
\sup_{\lambda>0} \lambda^q \omega(\{x\in \Rn;\mathcal{M}_\alpha(\vec{f})(x)>\lambda\}).
\end{equation}
\end{proposition}
\remark
When $m=1$, Theorem \rm{1} \cite{MW} is the linear result of Proposition \ref{pro 5.1}.
Furthermore, the inequality (\ref{strong estimate}) has been proved by Moen in Theorem \rm{3.1} \cite{Moen},
which used an extrapolation theorem. Our method is entirely different from his.

To prove Proposition \ref{pro 5.1}, we need the following lemma.
\begin{lemma}\label{lemma 5.2}
If $0<\alpha<mn$, there exist constants $B$ and $K$, depending only on $\alpha$, $m$ and $n$
such that if $\lambda>0$, $d>0$, $b\geq B$, $\vec{f}$ is nonnegative, $Q$ is a cube in $\Rn$ such that
$\mathcal{I}_\alpha(\vec{f})\leq \lambda$ at some point of $Q$ and
$E = \{ x \in Q;\mathcal{I}_\alpha(\vec{f})(x)
\geq \lambda b, \mathcal{M}_\alpha(\vec{f})(x) \leq \lambda d \}$,
then $|E|\leq K |Q| (d/b)^{n/(mn-\alpha)}$.
\end{lemma}
\begin{proof}
Here we need to use the endpoint unweighted estimate for $\mathcal{I}_\alpha$ in Lemma \rm{7} \cite{Kenig}.
Lemma \ref{lemma 5.2} is the multilinear version of Lemma \rm{1} \cite{MW}. Since the main ideas are almost the same as \cite{MW}, we omit the proof here.
\end{proof}

\begin{remark}
Lemma \ref{lemma 5.2} does not hold for the dyadic maximal operator
$\mathcal{M}_\alpha^d$

 We give an example in the case $m=1$.

Let $0<\alpha<n$ and $B>0$. Let $Q_1=[0,1]^n$ and $Q_{-1}=[-1,0]^n$. Set
\begin{equation*}
f(x)=\frac{\chi_{Q_{-1}}(x)}{|x|^{\alpha}},\text{and} \ \lambda
    =I_\alpha f((1,1,\dots,1)).
\end{equation*}
Then $I_\alpha f(0)=\int_{Q_{-1}}\frac{dy}{|y|^n}=+\infty$, and
\begin{equation}\label{eq:I-E-1}
|\{x\in Q_1; I_\alpha f(x)>b\}|>0\text{ for any }b>0.
\end{equation}
If a dyadic cube $Q$ contains $x\in Q_1$, then it must be contained in
$[0,\infty)^n$. Since $\supp f\subset (-\infty,0]^n$, we get
\begin{equation*}
\mathcal{M}_\alpha^d({f})(x)
=\sup_{Q\ni x} \frac{1}{|Q|^{1-\frac{\alpha}{n}}}\int_Q |f(y)| \;dy
=0 \ \text{for} \ x \in Q_1.
\end{equation*}
So, for each $d>0$, $b\ge B$ we have by (\ref{eq:I-E-1})
\begin{equation}\label{eq:I-E-2}
|\{x\in Q_1; I_\alpha f(x)>b\lambda,
\,\mathcal{M}_\alpha^d({f})(x)\le d\lambda\}|
=|\{x\in Q_1; I_\alpha f(x)>b\lambda \}|>0.
\end{equation}
Therefore there exists no $K>0$ such that
\begin{equation*}
|\{x\in Q_1; I_\alpha f(x)>b\lambda,
\,\mathcal{M}_\alpha^d({f})(x)\le d\lambda\}|
\leq K |Q_1| (d/b)^{n/(n-\alpha)},
\end{equation*}
although $\lambda =I_\alpha f((1,1,\dots,1))$.
\end{remark}
\vspace{0.25cm}
\noindent\textbf{Proof of Proposition 5.1}\\

\begin{proof} The following argument is essentially taken from \cite{MW}.
Without loss of generality, we can assume that $\vec{f}$ is nonnegative and
has compact support.
For given $\lambda>0$, in the light of Whitney decomposition (Theorem $1$,
\cite[p.~167]{Stein}), there are cubes $\{Q_j\}$ with disjoint interiors such
that,
$$
\{x\in \Rn; \mathcal{I}_\alpha(\vec{f})>\lambda \}=\bigcup_{j=1}^\infty Q_j,
$$
and for each $j$, $\mathcal{I}_\alpha(\vec{f})\leq \lambda$ at some point of
$4Q_j$.
Let $B$ and $K$ be as in Lemma \ref{lemma 5.2} and let $b = \max(1, B)$.
As a property of $A_{\infty}$ weight,
it is known that for any $0<\varepsilon <1$ there exists $\delta>0$,
such that $|S|<\delta|Q|$ implies $\omega(S)<\varepsilon \omega(Q)$ for any
cube $Q$ and its measurable subset $S$. Let $\delta$ correspond to
$\varepsilon=\frac12 b^{-q}$ for $\omega(x)$. Choose $D$ so that
$\delta=K4^n(D/b)^{n/(mn-\alpha)}$.
Let $d$ satisfy $0<d \leq D$ and
$$
E_j=\{x\in Q_j; \mathcal{I}_\alpha(\vec{f})> \lambda b,
\mathcal{M}_\alpha(\vec{f})\leq \lambda d \}.
$$
According to Lemma \ref{lemma 5.2}, we have
$|E_j| \leq K |4Q_j| (d/b)^{n/(mn-\alpha)} < \delta |Q_j|$, and hence,
we have $\omega(E_j)\leq \frac12 b^{-q}\omega(Q_j)$. So,
\begin{align*}
&\omega(\{x\in \Rn; \mathcal{I}_\alpha(\vec{f})(x)> \lambda b,
\mathcal{M}_\alpha(\vec{f})(x)\leq \lambda d  \} )
\\
&=\sum_{j=1}^\infty \omega( \{x\in \Rn; \mathcal{I}_\alpha(\vec{f})(x)
> \lambda b, \mathcal{M}_\alpha(\vec{f})(x)\leq \lambda d  \} \cap Q_j)
\\
&=\sum_{j=1}^\infty \omega(E_j)
\leq \frac12 b^{-q}\omega(\{x\in \Rn; \mathcal{I}_\alpha(\vec{f})(x)
> \lambda \}).
\end{align*}
Therefore we obtain
\begin{equation}\aligned \label{equation 5.5}
\omega(\{x\in \Rn; \mathcal{I}_\alpha(\vec{f})(x)> \lambda b \})
&\leq \omega(\{x\in \Rn; \mathcal{M}_\alpha(\vec{f})(x)> \lambda d \})
\\
&+ \frac12 b^{-q}\omega(\{x\in \Rn; \mathcal{I}_\alpha(\vec{f})(x)
> \lambda \}).
\endaligned
\end{equation}

Because $\vec{f}$ has compact support, there exists a cube $Q$ such that
$\vec{f}=0$ for any $x$ outside $Q$.
For fixed $x$ outside $3Q$, let $x_0$ be the point in $Q$ closest to $x$, let
$P$ be the smallest cube with center at $x$ and
sides parallel to $Q$ that contains $Q$. Then there is a constant $L=L(n)>1$
such that $|P|\leq L |x-x_0|^n$. Moreover,
\begin{align*}
\mathcal{I}_\alpha(\vec{f})(x)
&\leq \frac{1}{(m|x-x_0|)^{mn-\alpha}}\prod_{i=1}^m \int_Q f_i(y_i)dy_i
\leq \left( \frac{L}{m^n}\right)^{m-\alpha/n}\mathcal{M}_\alpha(\vec{f})(x).
\end{align*}
Hence, taking $d = \min\{D,(L/m^n)^{\alpha/n-m}\}$, we get
\begin{equation}\label{equation 5.6}
\{x\in \Rn; \mathcal{I}_\alpha(\vec{f})(x)> \lambda \} \cap (3Q)^{c}
\subset \{x\in \Rn; \mathcal{M}_\alpha(\vec{f})(x)> \lambda d \}.
\end{equation}
From (\ref{equation 5.5}) and (\ref{equation 5.6}), it follows that
\begin{equation}\aligned \label{equation 5.7}
&\omega(\{x\in \Rn; \mathcal{I}_\alpha(\vec{f})(x)> \lambda b \})
\\
&\leq \omega(\{x\in \Rn; \mathcal{M}_\alpha(\vec{f})(x)>  \lambda d \})
+ \frac12 b^{-q}\omega(\{x\in \Rn;\mathcal{I}_\alpha(\vec{f})(x)> \lambda
\}\cap(3Q)^{c})
\\
&\quad + \frac12 b^{-q}\omega(\{x\in \Rn;\mathcal{I}_\alpha(\vec{f})(x)
> \lambda \}\cap (3Q))
\\
&\leq 2\omega(\{x\in \Rn; \mathcal{M}_\alpha(\vec{f})(x)>  \lambda d \})
+ \frac12 b^{-q}\omega(\{x\in 3Q;\mathcal{I}_\alpha(\vec{f})(x)> \lambda \}).
\endaligned
\end{equation}

{\rm(i)}.
Let $N$ be any positive number, multiply both sides of (\ref{equation 5.7}) by
$\lambda^{q-1}$ and integrate with respect to $\lambda$ from $0$ to $N$,
then make a change of variables, we obtain
\begin{align*}
&b^{-q}\int_0^{bN}\lambda^{q-1}\omega(\{\mathcal{I}_\alpha(\vec{f})(x)
>  \lambda \})d\lambda
\\
&\leq 2 \int_0^{N}\lambda^{q-1}\omega(\{\mathcal{M}_\alpha(\vec{f})(x)
>  \lambda d\})d\lambda + \frac{b^{-q}}{2} \int_0^{N}\lambda^{q-1}
\omega(\{\mathcal{I}_\alpha(\vec{f})(x)>  \lambda \}\cap 3Q)d\lambda
\\
&\leq 2d^{-q} \int_0^{dN}\lambda^{q-1}
\omega(\{ \mathcal{M}_\alpha(\vec{f})(x)>  \lambda \})d\lambda
+ \frac{b^{-q}}{2} \int_0^{bN}\lambda^{q-1}
\omega(\{ \mathcal{I}_\alpha(\vec{f})(x)>  \lambda \})d\lambda.
\end{align*}
Therefore
$$
b^{-q}\int_0^{bN}\lambda^{q-1}\omega(\{\mathcal{I}_\alpha(\vec{f})(x)
>  \lambda \})d\lambda
\leq 4d^{-q} \int_0^{dN}\lambda^{q-1}\omega(\{\mathcal{M}_\alpha(\vec{f})(x)
>  \lambda \})d\lambda.
$$
Observe that $\big\|f\big\|_{L^q(\omega)}=q\int_0^{\infty}\lambda^{q-1}
\omega(\{x\in \Rn; |f(x)|>\lambda\})d\lambda$, $0<q<\infty$.
Finally, letting $N$ approach $\infty$, we deduce that
$$
\int_\Rn |\mathcal{I}_\alpha(\vec{f})(x)|^q\omega(x)dx \leq
4\left(\frac{b}{d}\right)^q \int_\Rn [\mathcal{M}_\alpha(\vec{f})(x)]^q
\omega(x)dx.
$$
This shows (\ref{strong estimate}).

\rm{(ii)}. Next we shall check (\ref{weak estimate}).
The technique is similar to \rm{(i)}.
Let $N$ be any positive number, multiply both sides of (\ref{equation 5.7})
by $\lambda^q$, then take the supremum of both sides for $0<\lambda<N$
and note the fact that $\sup(u+v)\leq \sup u + \sup v$. Then making a change
of variables, we have
\begin{align*}
&b^{-q}\sup_{0< \lambda < bN}\lambda^q
\omega(\{\mathcal{I}_\alpha(\vec{f})(x)>  \lambda \})
\\
&\leq 2 d^{-q}\sup_{0< \lambda < dN}\lambda^q
\omega(\{ \mathcal{M}_\alpha(\vec{f})(x)>  \lambda \})
+ \frac{b^{-q}}{2} \sup_{0<\lambda<N} \lambda^q
\omega(\{\mathcal{I}_\alpha(\vec{f})(x)>  \lambda \}\cap 3Q)
\\
&\leq 2 d^{-q}\sup_{0< \lambda < dN}\lambda^q
\omega(\{ \mathcal{M}_\alpha(\vec{f})(x)>  \lambda \})
+ \frac{b^{-q}}{2}\sup_{0< \lambda < bN}\lambda^q
\omega(\{\mathcal{I}_\alpha(\vec{f})(x)>  \lambda \}).
\end{align*}
Thus
$$
b^{-q}\sup_{0< \lambda < bN}\lambda^q
\omega(\{\mathcal{I}_\alpha(\vec{f})(x)>  \lambda \})
\leq 4 d^{-q}\sup_{0< \lambda < dN}\lambda^q
\omega(\{\mathcal{M}_\alpha(\vec{f})(x)>  \lambda \}).
$$
Now letting $N$ tend to $\infty$, we obtain
$$
\sup_{\lambda>0} \lambda^q
\omega(\{x\in \Rn;|\mathcal{I}_\alpha(\vec{f})(x)|>\lambda\})
\leq 4\left(\frac{b}{d}\right)^q \sup_{\lambda>0} \lambda^q
\omega(\{x\in \Rn;\mathcal{M}_\alpha(\vec{f})(x)>\lambda\}).
$$
This shows (\ref{weak estimate}).
So far, we have finished the proof of Proposition \ref{pro 5.1}.
\end{proof}


\end{document}